\documentclass[12pt]{amsart}
\usepackage[all, cmtip]{xy}

\usepackage[colorlinks=true]{hyperref}

\usepackage{times,upref,eucal, mathrsfs, xcolor, dsfont}
\usepackage{amsfonts, amscd, amsmath,amstext,amsbsy,amssymb, amsopn,amsthm}

\usepackage{url}

\usepackage{inputenc}

\usepackage{mathtools}
\usepackage{bookmark}

\theoremstyle{definition}
\newtheorem{defn}{Definition}
\newtheorem{thm}[defn]{Theorem}
\newtheorem{lem}[defn]{Lemma}

\newtheorem{prop}[defn]{Proposition}
\newtheorem{rem}[defn]{Remark}

\newtheorem{open}{Open Problem}

\newcommand{\om}{\otimes_{min}}
\newcommand{\rank}{\texttt{rank\,}}
\newcommand{\n}{\|}
\newcommand{\ncb}{\|_{cb}}
\input xy
\xyoption{all}

\begin{document}

\title{Duality of uniform approximation property in operator spaces}

\author{Yanqi Qiu} 
\address{Yanqi Qiu: Aix-Marseille Universit{\'e}, Centrale Marseille, CNRS, I2M, UMR7373,  39 Rue F. Juliot Curie 13453, Marseille}
\email{yqi.qiu@gmail.com}

\begin{abstract}
The duality of uniform approximation property for Banach spaces is well known. In this note, we establish, under the assumption of local reflexivity,  the duality of uniform approximation property in the category of operator spaces.
\end{abstract}

\maketitle

\section{Introduction}
We say that an operator space $E$ has the uniform approximation property in operator space sense (in short OUAP), if there is a constant $K \geq 1$ and a function $k(n)$ such that , for any $n$-dimensional subspace $M$ of $E$, there exists a finite rank operator $T \in CB(E)$, such that $$\n T \ncb \leq K, \rank T \leq k(n) \text{ and } T|_M = id_M.$$ We will say that $E$ has $( K, k(n) )$ - OUAP, if the property holds for a constant $K$ and a function $k(n)$.

In this note, we show that OUAP pass to the dual under a milder condition. 

\begin{thm}\label{mainthm}
If $E$ (resp. $E^*$) has the $(K,k(n))$-OUAP, and $E^*$ (resp. $E$) is a locally reflexive operator space, then $E^{*}$ (resp. $E$) has the $$\left( \frac{1}{1-1/m} \left([(1+\varepsilon)K]^{1+m} + \frac{1}{m}\right), m^{2/m}[(1+\varepsilon)K]^{2 +2/m}k(n)^{1+1/m}\right)\text{-}\textrm{OUAP},$$ for all $\varepsilon > 0$ and all integers $m >1$. 
\end{thm}

For simplicity,  the locally reflexive in this note will always mean locally reflexive with constant 1. However, after a suitable modification of constants,  Theorem \ref{mainthm} still holds if we use locally reflexive with constant $\lambda > 1$.
 
It is not known whether we can drop the assumption on the local reflexivity in Theorem \ref{mainthm}. We formulate it as Open Problem \ref{o1} .

\section{The main result}
Given an operator ideal norm $\alpha$, we say that an operator space $E$ has $\alpha$ - OUAP, if in the definition of OUAP, the condition $\rank T \leq k(n)$ is replaced by the condition $\alpha(T) \leq k(n)$. We will say that $E$ has  $( K, k(n) )$ - $\alpha$ - OUAP, if the property holds for a constant $K$ and a function $k(n)$.

Let $E$ be an operator space and let $Y$ be a Banach space. Recall that an operator $u: E \rightarrow Y$ is called $(2,oh)$-summing if there is a constant $C$ such that for all finite sequences $(x_i)$ in $E$, we have
\begin{displaymath}
\sum_i \n u(x_i)\n^2 \leq C^2 \Big \n \sum_i x_i \otimes \bar{x_i} \Big \n_{E \om \overline{E}},
\end{displaymath}
and we denote by $\pi_{2,oh}(u)$ the smallest constant $C$ for which this holds.

Given an operator ideal norm, we define $\alpha^d$ the dual ideal norm by $$\alpha^d(T) = \alpha (T^*).$$ The operator ideal norm $\alpha$ is said to be 1-injective, if for any operator $u : E \rightarrow F$ and any completely isometric inclusion $i: F \hookrightarrow G,$ we have $$\alpha(T) = \alpha(i \circ T).$$ 

For an operator $T: E \rightarrow F$ and any integer $i \geq 1$, the $i$-th complete approximation number $b_i(T)$ of $T$ is defined by
\begin{displaymath}
b_i(T) = \inf\{\n T - S \ncb : S \in CB(E, F), \rank S < i\}.
\end{displaymath}
\begin{rem}
If $E$ is a homogeneous operator space, i.e., for all $T: E \rightarrow E$, we have $\n T \ncb = \n T \n$, then $b_i(T) = a_i(T)$, where $a_i(T)$ stands for the usual $i$-th approximation number of $T$. In particular, since the Piser's operator Hilbert space $OH$ is homogeneous, we have $b_i(T) = a_i(T)$ for any $T \in CB(OH) = B(OH)$.
\end{rem}

Let us recall the notion of locally reflexivity for operator spaces (see \cite{OS}). An operator space $E$ is called locally reflexive, if for any finite-dimensional operator space $L$, the natural linear isomorphism
\begin{displaymath}
CB(L, E^{**}) \rightarrow CB(L, E)^{**}
\end{displaymath}
is isometric.


The following lemma is an immediate generalisation of lemma 1 in the article \cite{UAP}. 
\begin{lem}\label{dual_lemma}
Let $\alpha$ be an 1-injective operator ideal norm. If $E$ be a locally reflexive operator space, and $E^*$ has $\Big( K, k(n)\Big)$ - $\alpha^d$ - OUAP, then $E$ has $\Big( K(1+\varepsilon), k(n)(1+ \varepsilon)\Big)$ - $\alpha$ - OUAP, for all $\varepsilon > 0$.
\end{lem}

\begin{proof}
Assume $E^*$ has $(K,k(n))$ - $\alpha^d$ - OUAP. Let $M$ be an $n$-dimensional subspace of $E$, fix $(e_1, \cdots, e_n)$ an Auberbach basis of $M$, i.e., a basis such that
\begin{displaymath}
\max_i | \lambda_i | \leq \| \sum_i \lambda_i e_i \| \leq \sum_i | \lambda_i |
\end{displaymath}
for all scalars $\lambda_i$. With the dual basis, it is easy to see that  
\begin{displaymath}
\max_i \n a_i \n \leq \n \sum_i a_i \otimes e_i \n_{min} \leq \sum_i \n a_i \n
\end{displaymath}
for all elements $a_i$ in some operator space $G$. Fix $\varepsilon > 0$ and define
\begin{align*}
\mathscr{R}  & = \bigg \{ T \in CB(E): \n T \ncb \leq K(1+\varepsilon)^{1/2}, \\ & \qquad \qquad \alpha(T) \leq k(n)(1+ \varepsilon)^{1/2},     \rank T < \infty \bigg\}
\end{align*}
and
\begin{displaymath}
\mathscr{C} = \bigg \{ (Te_i, \cdots, Te_n) : T \in \mathscr{R} \bigg \} \subset \ell_{\infty}^n( E ).
\end{displaymath}
We claim first that $(e_1, \cdots, e_n) \in \overline{\mathscr{C}}$, the norm closure of $\mathscr{C}$ in $\ell_{\infty}^n( E)$. Otherwise, since $\mathscr{C}$ is convex, and $\ell_{\infty}^n(E)$ has as dual space $\ell_1^n(E^*)$, by Hahn-Banach separating theorem, there exist $\xi_1, \cdots, \xi_n$ in $E^*$, such that
\begin{displaymath}
\sum_i(\xi_i, Te_i) < \sum_i (\xi_i, e_i), \quad \forall T \in \mathscr{R}.
\end{displaymath}
Since $E^*$ has $\Big( K, k(n)\Big)$ - $\alpha^d$ - OUAP, we can find an finite rank operator $S \in CB(E^*)$, such that $$\n S\ncb \leq K, \alpha^d(S) \leq k(n) \quad\text{and} \quad S\xi_i = \xi_i \text{\, for all\, } i = 1, \cdots, n.$$ Since $E$ is locally reflexive, the range of $S^*$ is a finite dimensional subspace  $R(S^*)$ of $E^{**}$, we can find an operator $\varphi : R(S^*) \rightarrow E$, such that $$\n \varphi \ncb \leq (1+ \varepsilon)^{1/2}$$  and $$ (\varphi(x), \xi_i) = (x, \xi_i) \text{ for all } i = 1, \cdots, n \text{ and } x \in R(S^*).$$ Let us denote by $\overline{S^*}$ when $S^*$ is considered as an operator $E^{**} \rightarrow R(S^*)$. Since $\alpha$ is 1-injective, $$\alpha(\overline{S^*}) = \alpha (S^*) = \alpha^d (S) \leq k(n).$$ Let $T_0$ be the composition of the following applications: 
\begin{displaymath}
\begin{CD}
T_0 : E @>i_E>> E^{**} @>\overline{S^*}>> R(S^*) @>\varphi>> E,
\end{CD}
\end{displaymath}
where $i_E$ is the canonical inclusion. We have $$\n T_0 \ncb \leq \n i_E \ncb \n \overline{S^*} \ncb \n \varphi \ncb \leq K(1+ \varepsilon)^{1/2}$$ and $$\alpha(T_0) \leq \n i_E \ncb \alpha (\overline{S^*})\n \varphi \ncb \leq k(n)(1+\varepsilon)^{1/2},$$ consequently $T_0 \in \mathscr{R}$. Moreover $$(\xi_i, T_0e_i) = (\xi_i, \varphi(S^*(e_i)) = (\xi_i, S^*e_i) = (S\xi_i, e_i) = (\xi_i, e_i),$$ and hence $T_0$ satisfies
\begin{displaymath}
\sum_i(\xi_i, T_0e_i) = \sum_i (\xi_i, e_i),
\end{displaymath}
we get a contradiction.

Now we have proved $(e_1, \cdots, e_n) \in \overline{\mathscr{C}}$, for any $\mu > 0$, we can find $T \in \mathscr{R}$, such that $\n Te_i - e_i \n \leq \mu$. When $\mu$ is chosen to be small enough, $T|_M : M \rightarrow T(M)$ is invertible with inverse $V: T(M) \rightarrow M$. For any $n$-tuple $(a_i)$ in the operator space $ \mathscr{K}=\mathscr{K}(\ell_2)$, we have
\begin{eqnarray*}
\Big \n \sum_i a_i \otimes T(e_i) \Big\n_{min} &\geq &\Big\n \sum_i a_i \otimes e_i \Big\n_{min} - \Big\n \sum_i a_i \otimes (T(e_i) - e_i) \Big\n_{min} \\
& \geq &\Big\n \sum_i a_i \otimes e_i \Big\n_{min} - \mu \sum_i \n a_i \n \\
&\geq & \Big\n a_i \otimes e_i\Big \n_{min} - n\mu \sup_i \n a_i\n \\
&\geq &(1-n\mu) \Big\n \sum_i a_i \otimes e_i\Big \n_{min}, 
\end{eqnarray*}
which implies that $$id_\mathscr{K} \otimes V :  \mathscr{K} \om T(M) \rightarrow \mathscr{K} \om M$$ has norm less than $\frac{1}{1-n\mu}$, hence $$\n V \ncb \leq \frac{1}{1- n\mu}.$$ Let $P$ be a projection from $E$ onto $T(M)$, such that $\n P \ncb \leq n$, for example, let us denote by $(x_1, \cdots, x_n)$ an Auerbach basis for $T(M)$, and $(x_1^*, \cdots, x_n^*)$ its dual basis in $T(M)^*$, we can norm preservingly extend $x_i^*$, so that $x_i^*$ can be viewed as an element in $E^*$, then the projection $P$ defined by $$P e = \sum_i x_i^*(e)x_i, \quad \text{for all $e \in E$}$$ has c.b. norm less than $n$. We have the following commutative diagram: 
\begin{displaymath}
\begin{CD}
E @>T>> E = T(M) \oplus E_1 @>Q >> E\\
@AA{\texttt{inclusion}}A @VV {P}V @AA{\texttt{inclusion}}A\\
M @> T>> T(M) @> V>> M,
\end{CD}
\end{displaymath}
where $E_1 = \ker P$ and  $T(M) \oplus E_1$ is an algebraic direct sum, $Q$ is defined by $$Q = 1- P + VP.$$ Hence we have  $Q|_{T(M)} = V$ and $Q|_{E_1}$ is the inclusion of $E_1$ into $E$. Now let $F = QT$, then $$F|_M = id_M, \, \rank F \leq \rank T < \infty.$$ Let $J : T(M) \rightarrow E$ be the inclusion map and let $VP - P$ be the composition of the following maps: 
\begin{displaymath}
\begin{CD}
E @> P >> T(M) @> V -  J >> E.
\end{CD}
\end{displaymath} We have $$\n Q \ncb \leq 1 + \n VP - P \ncb.$$ Consider the map $$id_\mathscr{K} \otimes (V - J) : \mathscr{K}  \om T(M) \rightarrow \mathscr{K}  \om E.$$  We have  
\begin{align*}
\Big \n \sum_i a_i \otimes (e_i - Te_i) \Big \n_{min} &\leq \mu \sum_i \n a_i \n \leq n \mu \sup_i \n a_i \n  \\ &  \leq n\mu \Big \n \sum_i a_i \otimes e_i \Big \n_{min} \\
&\leq  n \mu \n V \ncb \Big \n \sum_i a_i \otimes Te_i \Big \n_{min} \\
& \leq  \frac{n\mu}{1-n\mu} \Big \n \sum_i a_i \otimes Te_i \Big \n_{min}, 
\end{align*}
which implies that $\n V - J \ncb \leq \frac{n \mu}{1 - n \mu}$. Hence $$\n Q \ncb \leq 1 + \frac{ n^2 \mu}{1 - n \mu},$$ when $\mu$ is small enough, we have $\n Q \ncb \leq (1 +\varepsilon)^{1/2}$, consequently we have $$\n F \ncb \leq K(1+ \varepsilon) \text{ and }\alpha(F) \leq k(n)(1 + \varepsilon).$$
\end{proof}


We now list some properties about $(2,oh)$-summing norm (see \cite{OH} p.88-p.89 for details).
\begin{enumerate}
\item[(i)] For any operator $u: OH \rightarrow E$ we have $$\pi_{2,oh}(u) = \pi_2(u).$$
\item[(ii)] Any operator $u: E \rightarrow OH$ which is $(2,oh)$-summing is necessarily completely bounded and we have $$\n u \ncb \leq \pi_{2,oh}(u).$$
\item[(iii)] Let $M$ be any $n$-dimensional operator space, then there is an isomorphism $u: M \rightarrow OH_n$, such that $$\pi_{2,oh}(u) = n^{1/2}, \quad \n u^{-1} \ncb = 1.$$
\end{enumerate} 

Let $E, F$ be two operator spaces. For any linear map $T: E \rightarrow F$, we define a number $\delta(T) \in [ 0, \infty ]$ as: 
\begin{displaymath}
\delta(T) = \inf\left\{ \n v \ncb \pi_{2, oh}(w) \right\}, 
\end{displaymath}
where the infimum runs over all possible factorizations of $T$ through some operator Hilbert space $OH(I)$ as following: 
\begin{align}\label{factorization}
\xymatrix{ &OH(I) \ar[dr]^v&\\
E \ar[ur]^w \ar[rr]^-T &&F
}.
\end{align}

\begin{prop}
$\delta$ is an 1-injective operator ideal norm. 
\end{prop}
\begin{proof}
If $T: E\rightarrow F$ has a factorization $T = vw$ as in \eqref{factorization} with $$\n v \ncb \pi_{2,oh}(w) < \infty,$$ then $$\n T \ncb \leq \n v \ncb \n w \ncb \leq \n v \ncb \pi_{2,oh}(w),$$ by definition of $\delta(T)$, we have $$\n T \ncb \leq \delta(T).$$ It is easy to verify that if  $$S: \xymatrix{L\ar[r]^{\alpha} & E \ar[r]^ T & F \ar[r]^{\beta} &G},$$ then we have $$\delta(S) = \delta(\beta T \alpha) \leq \n \beta \ncb \delta (T) \n \alpha \ncb.$$ Assume that $i: F \rightarrow G$ is an completely isometry, such that we have $$\xymatrix{&OH(I)\ar[dr]^v&\\ E\ar[ur]^w\ar[r]^T&F\ar@{^{(}->}[r]^i &G}.$$ Let $\overline{R(w)}$ be the closure of the range of $w$ in $OH(I)$, then there is some index set $J$ such that we have an identification $$\overline{R(w)} = OH(J)$$ completely isometrically. Now we define  $$\tilde{w}: E \rightarrow \overline{R(w)}=OH(J)$$ given by  $$\tilde{w}(e)  = w(e), \quad \text{for any $e \in E$ }.$$ Since $i \circ T = v \circ w = \tilde{v}\circ \tilde{w}$, the range of the $v|_{OH(J)}$ is contained in $F$, we denote by $\tilde{v}: OH(J) \rightarrow F$ the mapping given by $$\tilde{v}(x)  = v(x), \quad \text{for any $x \in OH(J).$} $$ Then $T = \tilde{v} \circ \tilde{w}$, so we find $$\delta(T) \leq \n\tilde{v}\ncb \pi_{2,oh}(\tilde{w}) \leq \n v\ncb \pi_{2,oh}(w),$$ and thus $\delta(T) \leq \delta(i\circ T)$. The inverse inequality has already been shown, thus $\delta$ is 1-injective.

We show now that $\delta$ satisfies the triangle inequality. Let $T_1,T_2 : E \rightarrow F$ be two operators with $\delta(T_1), \delta(T_2)$ finite. For any $\varepsilon > 0$, we can factorize $T_i$ as $$T_i: \xymatrix{E\ar[r]^{w_i}& OH(I_i) \ar[r]^{v_i}&F},$$ such that $$\n v_i \ncb = \pi_{2,oh}(w_i) \leq \sqrt{ \delta(T_i)+\varepsilon}, \quad \text{for} \quad  i= 1,2,$$ where $I_1$ and $I_2$ two disjoint index sets. We imbed  $OH(I_i)$ canonically into $OH(I_1\cup I_2)=OH(I_1) \oplus OH(I_2)$, and denote the inclusions by $$J_i: OH(I_i) \rightarrow OH(I_1\cup I_2).$$ Let $P_i$ denote the orthogonal projection from $OH(I_1\cup I_2)$ onto $OH(I_i)$ respectively. Then $$T_1 + T_2 = v_1w_1+v_2w_2=AB,$$ where $B: E \rightarrow OH(I_1 \cup I_2)$ is defined by $$B(x) = J_1w_1(x)+J_2w_2(x)$$ and $A: OH(I_1 \cup I_2) \rightarrow F$ is defined by $$A(y)= v_1J_1^{-1}P_1(y)+v_2J_2^{-1}P_2(y).$$ For all finite sequences $(x_i)$ in $E$, we have 
\begin{align*}
\sum \n B(x_i) \n^2 & =  \sum \n J_1w_1(x_i)+J_2w_2(x_i)\n^2 \\ & = \sum \n w_1(x_i)\n^2+\n w_2(x_i)\n^2 \\ &\leq  \left(\pi_{2,oh}(w_1)^2+\pi_{2,oh}(w_2)^2\right) \left\n \sum x_i\otimes \overline{x_i} \right\n_{E\om\overline{E}} \\
&\leq \left(\delta(T_1)+\delta(T_2)+2\varepsilon \right)\left\n x_i\otimes\overline{x_i}\right\n_{E\om\overline{E}}.
\end{align*}
So we have $$\pi_{2,oh}(B) \leq \sqrt{\delta(T_1)+\delta(T_2)+2\varepsilon}.$$ For the c.b. norm of $A$, assume that $(T_{i_1})_{i_1\in I_1}$ and $(T_{i_2})_{i_2 \in I_2}$ are normalised orthogonal basis for $OH(I_1)$ and $OH(I_2)$ respectively. Then
\begin{align*}
\n A\ncb^2 =& \sup\Bigg\{ \Big \n \sum_{i_1 \in J_1} A(T_{i_1})\otimes \overline{A(T_{i_1})}+\sum_{i_2\in J_2} A(T_{i_2})\otimes \overline{A(T_{i_2})} \Big\n_{F\om\overline{F}}: \\ &  \quad \qquad  \qquad J_1 \subset I_1,|J_1|<\infty; J_2 \subset I_2,|J_2|<\infty\Bigg\}\\
\leq & \sup_{J_1 \subset I_1, |J_1|<\infty}\left\n \sum_{i_1 \in J_1} v_1(T_{i_1})\otimes \overline{v_1(T_{i_1})} \right\n_{F\om\overline{F}} +\\ & +  \sup_{J_2 \subset I_2, |J_2|<\infty} \left\n \sum_{i_2\in J_2} v_2(T_{i_2})\otimes \overline{v_2(T_{i_2})} \right\n_{F\om\overline{F}}\\
=&\n v_1\ncb^2+\n v_2\ncb^2 \leq \delta(T_1)+\delta(T_2)+2\varepsilon.
\end{align*}
By the definition of $\delta$, we have $$\delta(T_1+T_2)\leq \delta(T_1)+\delta(T_2)+2\varepsilon$$ for any $\varepsilon$, hence we get $$\delta(T_1+T_2) \leq \delta(T_1)+\delta(T_2),$$ as desired.
\end{proof}

\begin{prop}
For any finite rank operator $T: E\rightarrow F$, we have $$\delta(T)\leq \n T\ncb\sqrt{\rank T}.$$
\end{prop}
\begin{proof}
We can factorize $T$ as following $$\xymatrix{E\ar[r]^T&R(T)\ar[r]^{id_{R(T)}}&R(T)\ar@{^{(}->}[r] &F}.$$ The property (iii) of the $(2,oh)$-summing norm gives that $$\delta(id_{R(T)})\leq \sqrt{\rank T}.$$ So we have $$\delta(T)\leq \n T \ncb \sqrt{\rank T}.$$ 
\end{proof}


\begin{rem}\label{equi}
If $E$ has the $(K,k(n))$-OUAP, then $E$ has the $$ (K,Kk(n)^{1/2})\text{-}\delta\text{-}\textrm{OUAP}$$ and also the $(K,Kk(n)^{1/2}$-$\delta^d$-OUAP. The following lemma shows that in fact the OUAP and the $\delta$-OUAP are equivalent.
\end{rem}

\begin{lem}\label{equi_lem}
If $E$ has $\Big(K,k(n)\Big)$-$\delta$-OUAP, then $E$ has $$\Big( \frac{1}{1-1/m}(1/m+K^{m+1}), m^{2/m}k(n)^{2+2/m} \Big)\text{-}\textrm{OUAP},$$ for all integers $m > 1$.
\end{lem}
\begin{rem}
For simplification, here we replace the inequality $\delta(T) \leq k(n)$ in the definition of $\big(K,k(n)\big)$-$\delta$-OUAP by the strict inequality $\delta(T) < k(n)$,which of course is not an essential change.
\end{rem}
\begin{proof}
Assume $E$ has $\big(K, k(n) \big)$-$\delta$-OUAP. Fix an integer $m > 1$ and an $n$-dimensional subspace $M$ of $E$. Then we can find a finite rank operator $T: E \rightarrow E$, such that $$T|_E = id_E, \quad \n T \ncb \leq K \quad  \text{and} \quad \delta(T) < k(n).$$ By the definition of $\delta(T)$, we can factorize $T$ as: 
\begin{displaymath}
\xymatrix{
& OH \ar[dr]^ A&\\
E\ar[ur]^ B\ar[rr]^ T&& E
}
\end{displaymath}
such that $\pi_{2,oh}(B) < k(n)$ and $\n A \ncb \leq 1$. Since $$T^{m+1} = (AB)^ {m+1}=A(BA)^mB,$$ and $BA$ is an operator $OH \rightarrow OH$, we have 
\begin{align*}
b_i(T^{m+1})  & \leq \n A \ncb \n B \ncb b_i((BA)^m) \\ & = \n A \ncb \n B \ncb a_i((BA)^m) \\ & \leq \pi_{2,oh}(B)a_i((BA)^m).
\end{align*}
The sequence $(b_i(T))_{i\geq 1}$ is nonincreasing, so we have:
\begin{eqnarray*}
sup_i i^{m/2}b_i(T^{m+1}) & \leq & \left(\sum_i b_i(T^{m+1})^{2/m} \right)^{m/2}\\
& \leq & \pi_{2,oh}(B)\left(\sum_i a_i((BA)^m)^{2/m}\right)^{m/2}\\
& = & \pi_{2,oh}(B)\n(BA)^m\n_{S_{2/m}}\\
& \leq &\pi_{2,oh}(B)\n BA \n_{S_2}^m\\
& = & \pi_{2,oh}(B) \pi_2(BA)^m\\
& = & \pi_{2,oh}(B) \pi_{2,oh}(BA)^m\\
& \leq & \pi_{2,oh}(B)^{m+1}\\
& < & k(n)^{m+1},
\end{eqnarray*} 
where we have used the facts that the 2-summing norm and the Hilbert-Schmidt norm coincide for operators between to Hilbert spaces, the $(2,oh)$-summing norm and the 2-summing norm for operators from a Piser's operator Hilbert space $OH$ to some other Banach space coincide. Let $i_0$ be the smallest integer strictly greater than $m^{2/m}k(n)^{2+2/m}$, then $i_0^{m/2} \geq m k(n)^{m+1}$, so we have $b_{i_0}(T^{m+1}) < 1/m$. By the definition of $b_i(T)$, there exists $S: E \rightarrow E$, such that $$ \rank S < i_0 \quad \text{and}  \quad \n T^{m+1} - S\ncb < 1/m.$$ This implies that $$\rank S \leq m^{2/m}k(n)^{2+2/m}$$ and  that $id_E -T^{m+1} + S$ is invertible with an inverse $V$, whose c.b. norm satisfies $$\n V \ncb < \frac{1}{1-1/m}.$$ Consequently, if we define $$T_0 = VS: E \rightarrow E,$$ then $$\text{ $T_0|_M = id_M$ \,  and \, $\rank T_0 \leq \rank S \leq m^{2/m}k(n)^{2+2/m}$.} $$  For the c.b. norm of $T_0$, we have 
\begin{align*}
\n T_0 \ncb & \leq \frac{1}{1-1/m}\left(\n S - T^{m+1}\ncb + \n T^{m+1}\ncb \right)  \\ & \leq \frac{1}{1-1/m}\left(1/m + K^{m+1}\right),
\end{align*} 
this is exactly what we want.
\end{proof}

We will use the following proposition (cf.\cite{ultra_C_alg}).
\begin{prop}\label{ultra}
For an operator space $E$, there are an infinite set $I$ and a non-trivial ultrafilter $\mathscr{U}$ on $I$, a completely isometric embedding $j: E^{**} \rightarrow E^I/\mathscr{U}$, and $j(E^{**})$ is completely complemented in $E^I/\mathscr{U}$ (i.e. there is a completely contractive surjective projection $P: E^I/\mathscr{U} \rightarrow j(E^{**})$), such that we have the following commutative diagram:
$$\xymatrix{
E \ar[rr]^i\ar[dr]^{i_E} &&E^I/\mathscr{U}\\
&E^{**} \ar[ur]^-{j}
},$$ where $i$ and $i_E$ are canonical inclusions.
\end{prop}

\begin{prop}
The class of operator spaces having the $(K,k(n))$-OUAP is stable under ultraproducts. In particular, if $E$ has the $(K,k(n))$-OUAP,then so does $E^{**}$.
\end{prop}
\begin{proof}
Let $(E_i)_{i\in I}$ be a family of operator spaces having the $(K,k(n))$-OUAP, $\mathscr{U}$ an ultrafilter on $I$. We want to show that $\Pi_{i \in I } E_i/\mathscr{U}$ has the $(K,k(n))$-OUAP. For any $n$-dimensional subspace $$M \subset \Pi_{i \in I} E_i/\mathscr{U},$$ choose an algebraic basis $x^1,\cdots,x^n$ of $M$, with $x^k=(x^k_i)_{\mathscr{U}}$. Let $M_i$ be the linear span of $x^k_i$ for $k=1,\cdots, n$, obviously, we have $$M = \Pi_{i \in I } M_i/\mathscr{U}.$$ Since each $E_i$ has the $(K,k(n))$-OUAP, we can find $T_i : E_i \rightarrow E_i$ such that $$\n T_i\ncb \leq K,\quad \rank T_i \leq k(n) \quad \text{ and}  \quad T_i|_{M_i} = id_{M_i}.$$ Let $$T = (T_i)_{\mathscr{U}}: \Pi_{i \in I } E_i/\mathscr{U} \rightarrow \Pi_{i \in I } E_i/\mathscr{U},$$ then $$\n T\ncb \leq \lim_{\mathscr{U}} \n T_i \ncb \leq K, \, \rank T \leq k(n), \,T|_M = id_M.$$ According to Proposition \ref{ultra}, since $E^{**}$ is completely complemented in some ultrapower of $E$, it is easy to show $E^{**}$ has the $(K,k(n))$-OUAP when $E$ has it.
\end{proof}

\begin{proof}[Proof of Theorem \ref{mainthm}]
Assume that $E$ has the $(K,k(n))$-OUAP, then so does $E^{**}$. As in Remark \ref{equi}, $E^{**}$ has the $\Big(K,Kk(n)^{1/2}\Big)$-$\delta^d$-OUAP. If $E^*$ is locally reflexive, and since $\delta$ is 1-injective, then we can apply Lemma \ref{dual_lemma} to show that $E^*$ has $$\Big( K(1+\varepsilon), Kk(n)^{1/2}(1+ \varepsilon)\Big) \text{-} \delta \text{-} \textrm{OUAP},$$ for all $\varepsilon > 0$. Now by applying  Lemma \ref{equi_lem}, and get the desired result. The case from $E^*$ to $E$ is more direct without the argument of ultraproducts.
\end{proof}

It seems to be interesting to ask whether we can drop the assumption on local reflexivity  in Theorem \ref{mainthm}. The following question seems to be open. 

\begin{open}\label{o1}
Does the OUAP property of  $E$ (resp. $E^*$) imply that $E^*$ (resp. $E$) is locally reflexive? 
\end{open}

The above open problem is related to the following result of Ozawa, see section 4 of \cite{Ozawa-neb}.

\begin{prop}\label{ozawa}(Ozawa)
The CBAP property does not imply locally reflexivity.
\end{prop}

\begin{rem}
After writing this note, the author was told by Pisier that in fact the ideal norm $\delta$ defined here coincides with the completely 2-summing norm $\pi_2^\circ$ (cf. \cite{Lp_non_comm}, p.62). 
\end{rem}

\section*{acknowledgment}
I would like to thank Gilles Pisier for inviting me to Texas A \& M university and for the valuable conservations on many different subjects. I am grateful to  Narutaka Ozawa for telling me that ''CBAP does not imply local reflexive''.  I would like to thank Rui Liu for his interests in this work and for many valuable discussions. I also would like to  thank Issan Patri for discussions in mathematics and cultures in India.

The author is supported by A*MIDEX project (No. ANR-11-IDEX-0001-02) and partially supported by the ANR grant 2011-BS01-00801.


\end{document}